\documentclass[authoryear,1p]{elsarticle}
\usepackage{listings}
\usepackage{color}
\usepackage[]{multicol}
\usepackage{amsthm}
\usepackage{multirow}
\usepackage{caption}
\usepackage{here}
\usepackage{bm}
\usepackage{mathrsfs}
\usepackage{subcaption}
\usepackage{amsmath, amssymb}
\usepackage{ascmac}
\usepackage{url}
\usepackage{comment}
\usepackage{nccmath}
\usepackage {booktabs}
\usepackage{natbib}

\newtheorem{theorem}{Theorem}

\newtheorem{lemma}{Lemma}
\newtheorem{algorithm}{Algorithm}

\newtheorem{remark}{Remark}
\def\vec#1{\mbox{\boldmath{$#1$}}}

\usepackage{algorithm}
\usepackage{algorithmic}
\usepackage{tcolorbox}

\journal{}
\begin{document}
%
%
\def\tr{\mathrm{tr}}              
\def\etr{\mathrm{etr}}            
\def\Var{\mathrm{Var}}            
\def\Cov{\mathrm{Cov}}            
\def\E{\mathrm{E}}                
\def\zonal{\boldsymbol{Z}}        
\def\Vec{\mathrm{Vec}}            
\def\Pr{\mathrm{Pr}}              
\def\bnull{\boldsymbol{0}}        
\def\bones{\boldsymbol{1}}        
\def\bpar{\boldsymbol{\partial}}  
\def\diag{\mathrm{diag}}          
%
%
\def\bfa{\boldsymbol{a}}
\def\bfb{\boldsymbol{b}}
\def\bfc{\boldsymbol{c}}
\def\bfd{\boldsymbol{d}}
\def\bfe{\boldsymbol{e}}
\def\bff{\boldsymbol{f}}
\def\bfg{\boldsymbol{g}}
\def\bfh{\boldsymbol{h}}
\def\bfi{\boldsymbol{i}}
\def\bfj{\boldsymbol{j}}          
\def\bfk{\boldsymbol{k}}
\def\bfl{\boldsymbol{l}}
\def\bfm{\boldsymbol{m}}
\def\bfn{\boldsymbol{n}}
\def\bfo{\boldsymbol{o}}
\def\bfp{\boldsymbol{p}}
\def\bfq{\boldsymbol{q}}
\def\bfr{\boldsymbol{r}}
\def\bfs{\boldsymbol{s}}
\def\bft{\boldsymbol{t}}
\def\bfu{\boldsymbol{u}}
\def\bfv{\boldsymbol{v}}
\def\bfw{\boldsymbol{w}}
\def\bfx{\boldsymbol{x}}
\def\bfy{\boldsymbol{y}}
\def\bfz{\boldsymbol{z}}
%
%
\def\bfA{\boldsymbol{A}}
\def\bfB{\boldsymbol{B}}
\def\bfC{\boldsymbol{C}}
\def\bfD{\boldsymbol{D}}
\def\bfE{\boldsymbol{E}}
\def\bfF{\boldsymbol{F}}
\def\bfG{\boldsymbol{G}}
\def\bfH{\boldsymbol{H}}
\def\bfI{\boldsymbol{I}}
\def\bfJ{\boldsymbol{J}}
\def\bfK{\boldsymbol{K}}
\def\bfL{\boldsymbol{L}}
\def\bfM{\boldsymbol{M}}
\def\bfN{\boldsymbol{N}}
\def\bfO{\boldsymbol{O}}
\def\bfP{\boldsymbol{P}}
\def\bfQ{\boldsymbol{Q}}
\def\bfR{\boldsymbol{R}}
\def\bfS{\boldsymbol{S}}
\def\bfT{\boldsymbol{T}}
\def\bfU{\boldsymbol{U}}
\def\bfV{\boldsymbol{V}}
\def\bfW{\boldsymbol{W}}
\def\bfX{\boldsymbol{X}}
\def\bfY{\boldsymbol{Y}}
\def\bfZ{\boldsymbol{Z}}
%
%
\def\balpha{\boldsymbol{\alpha}}
\def\bbeta{\boldsymbol{\beta}}
\def\bgamma{\boldsymbol{\gamma}}
\def\bdelta{\boldsymbol{\delta}}
\def\bepsilon{\boldsymbol{\epsilon}}
\def\bvepsilon{\boldsymbol{\varepsilon}}
\def\bzeta{\boldsymbol{\zeta}}
\def\b-eta{\boldsymbol{\eta}}
\def\btheta{\boldsymbol{\theta}}
\def\bvtheta{\boldsymbol{\vartheta}}
\def\biota{\boldsymbol{\iota}}
\def\bkappa{\boldsymbol{\kappa}}
\def\blambda{\boldsymbol{\lambda}}
\def\bmu{\boldsymbol{\mu}}
\def\bnu{\boldsymbol{\nu}}
\def\bxi{\boldsymbol{\xi}}
\def\bpi{\boldsymbol{\pi}}
\def\bvpi{\boldsymbol{\varpi}}
\def\brho{\boldsymbol{\rho}}
\def\bvrho{\boldsymbol{\varrho}}
\def\bsigma{\boldsymbol{\sigma}}
\def\bvsigma{\boldsymbol{\varsigma}}
\def\btau{\boldsymbol{\tau}}
\def\bupsilon{\boldsymbol{\upsilon}}
\def\bphi{\boldsymbol{\phi}}
\def\bvphi{\boldsymbol{\varphi}}
\def\bchi{\boldsymbol{\chi}}
\def\bpsi{\boldsymbol{\psi}}
\def\bomega{\boldsymbol{\omega}}
%
%
\def\bGamma{\boldsymbol{\Gamma}}
\def\bDelta{\boldsymbol{\Delta}}
\def\bTheta{\boldsymbol{\Theta}}
\def\bLambda{\boldsymbol{\Lambda}}
\def\bXi{\boldsymbol{\Xi}}
\def\bPi{\boldsymbol{\Pi}}
\def\bSigma{\boldsymbol{\Sigma}}
\def\bUpsilon{\boldsymbol{\Upsilon}}
\def\bPhi{\boldsymbol{\Phi}}
\def\bPsi{\boldsymbol{\Psi}}
\def\bOmega{\boldsymbol{\Omega}}
%
\def\EC{\mathrm{EC}}       
\def\LS{\mathrm{LS}}       
\def\RS{\mathrm{RS}}       
\def\SS{\mathrm{SS}}       
\def\baleph{\mbox{\boldmath $\aleph$}} 
\def\bscriptaleph{\mbox{\boldmath {\scriptsize {$\aleph$}}}} 
\def\bfootnotealeph{\mbox{\boldmath {\footnotesize {$\aleph$}}}} 
\def\squareforqed{\hbox{\rlap{$\sqcap$}$\sqcup$}}
\def\qed{\ifmmode\else\unskip\quad\fi\squareforqed}
\def\smartqed{\def\qed{\ifmmode\squareforqed\else{\unskip\nobreak\hfil
\penalty50\hskip1em\null\nobreak\hfil\squareforqed
\parfillskip=0pt\finalhyphendemerits=0\endgraf}\fi}}
\begin{frontmatter}



\title{Multivariate normality test based on the uniform distribution on the Stiefel manifold
}


\author[label1]{Koki Shimizu}
\author[label2]{Toshiya Iwashita}

\address[label1]{Tokyo University of Science, 1-3 Kagurazaka, Shinjuku-ku, Tokyo, 162-8601, Japan}
\address[label2]{Tokyo University of Science, 2641 Yamazaki, Noda, Chiba, 162-8601, Japan}
\cortext[mycorrespondingauthor]{Corresponding author. Email address: \url{k-shimizu@rs.tus.ac.jp}~(K. Shimizu).}

\begin{abstract}
This study presents a new procedure for necessary tests of multivariate normality based on the uniform distribution on the Stiefel manifold.
We demonstrate that the test statistic, which is formed by the product of the scaled residual matrix and the symmetric square root of a Wishart matrix, is exactly distributed as a matrix-variate normal distribution under the null hypothesis.
Monte Carlo simulations are conducted to assess the Type I error rate and power in non-asymptotic settings.

\end{abstract}




\end{frontmatter}


\section{Introduction}
Let $\bfX$ be a $p$-dimensional random vector and $\bfX_1, \dots, \bfX_N$ be $N$ independent copies of $\bfX$.
We are interested in testing multivariate normality (MVN):
\begin{align}
\label{test-normality}
H_0 : \text{the law of } \bfX \text{ is } N_p(\boldsymbol{\mu}, \Sigma),
\end{align}
where $N_p(\boldsymbol{\mu}, \Sigma)$ denotes the
$p$-dimensional normal distribution with mean vector $\boldsymbol{\mu}$
and positive definite covariance matrix $\Sigma$.
To test \eqref{test-normality}, various tests based on properties of the normal distribution have been proposed; some well-known tests are summarized in \cite{henze2002invariant} and \cite{thode2002testing}.
A natural statistic for MVN is the scaled residuals (or the studentized residuals)
$$
\bfY_j=S^{-1/2}(\bfX_j-\bf\bar{X}),
$$
where
\begin{align}
\label{eq:mean-cov}
\bar{\boldsymbol{X}} = \frac{1}{N}\sum_{j=1}^N \bfX_j,
\quad
S = \frac{1}{n}\sum_{j=1}^N (\bfX_j - \bar{\boldsymbol{X}})(\bfX_j - \bar{\boldsymbol{X}})^{\prime},
\quad
n = N - 1 \ge p.
\end{align}
Here, the prime $(\cdot)^{\prime}$ and
$S^{-1/2}$ denote the transpose and inverse of a symmetric square root of $S$, respectively.
Based on \eqref{eq:mean-cov}, \cite{mardia1970,mardia1974} introduced multivariate skewness and kurtosis statistics and derived their asymptotic distributions.
\cite{baringhaus1988} considered a consistent test based on the discrepancy between the characteristic function of the multivariate normal distribution and that of the empirical distribution.
The test has asymptotic power against sequences of contiguous alternatives at the rate $N^{-1/2}$, regardless of the dimension in \cite{henze1997new}.
However, the scaled residuals $\bfY_j$ are not independent; thus, these are approximate tests that rely on the asymptotic independence of $\bfY_j$.
As a further test for normality, \cite{enomoto2020multivariate} examined test statistic based on normalizing transformation of the multivariate kurtosis.
The developments of the Baringhaus--Henze type tests are reviewed in detail in \cite{ebner2020}.

On the other hand, the joint exact distribution of the scaled residuals has been studied in \cite{iwashita2014,iwashita2020sut}.
\cite{iwashita2014} demonstrated that the density of
$Y=[\bfY_1, \ldots, \bfY_N]$ under spherical population is exactly distributed as a matrix variate left-spherical distribution.
By applying this result, \cite{iwashita2020sut} showed that a scaled residual matrix is uniformly distributed over the Stiefel manifold.

In this study, we propose a testing procedure for MVN based on a characterization of normality involving the uniform distribution over the Stiefel manifold.
In Section \ref{sec:02}, we introduce the distribution of a scaled residual matrix, which was provided in \cite{iwashita2020sut}.
In Section \ref{sec:03}, we show that the proposed test statistic for multivariate normality is exactly distributed as a matrix-variate normal distribution.
In Section \ref{sec:04}, we conduct numerical experiments to assess the performance of the proposed testing procedure and apply it to real data.

\section{Uniform Distribution on the Stiefel Manifold and its properties}
\label{sec:02}
We introduce random matrices, which are formed from the scaled residuals that are uniformly distributed over the Stiefel manifold.
The Stiefel manifold $\mathcal{O}(n,p)$ is defined as
\[
{\mathcal{O}}(n,p)
= \{ X \in \mathbb{R}^{n\times p} \mid X^\top X = I_p \}.
\]
If $n=p$, {\it{i.e}}., $\mathcal{O}(p,p)\equiv\mathcal{O}(p)$ coincides with the orthogonal group, and if $n=1$, it reduces to the hypersphere.
We write $X \sim \mathrm{LS}_{n \times p}(\phi_X)$
if $X$ has a matrix variate left-spherical distribution, that is, $X$ is a left-orthogonally invariant
$n \times p$ random matrix with characteristic function $\phi_X$.
The uniform distribution on the Stiefel manifold, which is denoted by $\mathcal{U}_{n,p}$,
is the distribution of a random matrix $X$ satisfying
$X \in \mathcal{O}(n,p)$ and $X \sim \mathrm{LS}_{n \times p}(\phi_X)$.
The statistical properties of distributions $\mathcal{U}_{n,p}$ have been studied in \cite{fang1990} and \cite{chikuse2003statistics}.

Let $Y$ be an $N \times p$ random matrix defined by $Y' = [\bfY_1,\dots,\bfY_N] = Q X' S^{-1/2}$,
where $X = [\bfX_1,\dots,\bfX_N]$, $Q = I_N - N^{-1}\mathbf{1}_N \mathbf{1}_N'$, \(\bones_{N} = [1,\ldots, 1]' \in \mathbb{R}^{N}\) and $\bfY_j$ are defined in \eqref{eq:mean-cov}.
Using the orthogonal projection decomposition of the centering matrix $Q$, \cite{iwashita2020sut} demonstrated that the $n\times p$ scaled residuals matrix
\begin{align}
\label{scaled residuals matrix}
Z = K'X'S^{-1/2},
\end{align}
which has a left-spherical distribution $\mathrm{LS}_{n\times p}(\phi_Z)$,
where $K$ is an $N\times n$ matrix satisfying $Q = KK'$, $K'K = I_n$, and
$K'\mathbf{1}_N = \mathbf{0}\in \mathbb{R}^{n}$.
Moreover, the \( n \times p \) matrix
\begin{align}
\label{matrix-Stiefel manifold}
U = Z(Z'Z)^{-1/2} = n^{-1/2} Z
\end{align}
is exactly distributed the uniform distribution over $\mathcal{O}(n,p)$.
From Theorem~1.5.5 of \cite{chikuse2003statistics}, the following lemma provides a characterization of normality for the uniform distribution over $\mathcal{O}(n,p)$.
\begin{lemma}
\label{lemma-chikuse}
Let $Z_0$ be an $n \times p$ random matrix.
Its unique polar decomposition of $Z_0$ is given by
\begin{align*}
Z_0 = H_{Z_0} T_{Z_0}^{1/2}, \quad
H_{Z_0} = {Z_0} ({Z_0}'{Z_0})^{-1/2}, \quad
T_{Z_0} = {Z_0}'{Z_0} .
\end{align*}
Then, the random matrix ${Z_0}$ is distributed as matrix-variate normal distribution $N(O, I_n\otimes I_p)$
(see, \cite{chikuse2003statistics}, p.~23 for details)
if and only if the following conditions are satisfied:
\begin{enumerate}
  \item $H_{Z_0}$ and $T_{Z_0}$ are independent,
  \item $H_{Z_0}$ is the uniformly distributed on $\mathcal{O}(n,p)$,
  \item $T_{Z_0}$ has the Wishart distribution $W_p(n, I_p)$.
\end{enumerate}
\end{lemma}

\section{Test statistic for multivariate normality and testing procedure}
\label{sec:03}
In this section, we derive the distribution of the test statistic based on the scaled residual matrix.
The following lemma shows that the distribution of the scaled residual matrix is invariant with respect to the mean vector and the covariance matrix.
\begin{lemma}
\label{lemma-IwashitaKlar}
Let \(\{\bfX_{j}\}_{j=1}^{N}\) be \(N\) independent random copies 
of \(\bfX \sim N_{p} (\bmu, \Sigma) \ (\Sigma >0)\) and \(X = [\bfX_{1},\ldots, \bfX_{N}] \sim N(\bmu \bones_{N}', 
\Sigma \otimes I_{N})\). 
Then 
\[
K'X'S^{-1/2} \stackrel{d}{=} K'{\tilde{X}}'\tilde{S}^{-1/2}
\]
where 
\(
	\tilde{X}=
	[\tilde{\bfX}_{1}, \ldots \tilde{\bfX}_{N}]=[\tilde{\bfX}_{(1)}, \ldots, \tilde{\bfX}_{(p)}]' 
	\sim N( O, I_{p} \otimes I_{N})
\), \(S\) and \(\tilde{S}\) are the sample covariance matrices 
based on \(\{\bfX_{j}\}_{j=1}^{N}\) and \(\{\tilde{\bf{X}}_{j}\}_{j=1}^{N}\),
respectively. 
The symbol ``$\overset{d}{=}$'' denotes equality in distribution.
\begin{proof}
Referring to the result in \cite{iwashita2014}, 
\(\Delta_{0} = \tilde{S}^{-1/2} \tilde{X}\) has a 
\textit{left-spherical distribution} \(\mathrm{\LS}_{ p\times N}(\phi_{\Delta_{0}})\).
Let \(\bfX^{*}_{j} = \bfX_{j} - \bmu \sim N_{p} (\bnull,
 \Sigma)\), 
 \(X^{*} = [\bfX^{*}_{1}, \ldots, \bfX^{*}_{N}] 
 \sim N(O, \Sigma \otimes I_{N} )\) and
 \[
 H_{\tilde{S}, \Sigma} = (\Sigma^{1/2} \tilde{S}\Sigma^{1/2})^{-1/2}
 \Sigma^{1/2} \tilde{S}^{1/2} \in \mathcal{O}(p).
 \]
Consider the characteristic function of 
\(\Delta = H_{\tilde{S},\Sigma} \Delta_{0}\), 
\(\phi_{\Delta}(T) \ (T \in \mathbb{R}^{p\times N})\). 
By making use of the normalized Haar measure, that is, 
\(\int_{\mathcal{O}(p)} (dH) = 1\), we have
\begin{align*}
	\phi_{\Delta} (T) &= \E 
	\left[
	\etr \left( i T'H_{\tilde{S}, \Sigma} \Delta_{0}\right)
	\right] \\
	&= \E
	\left[
	\int_{\mathcal{O}(p)} \etr \left(
	i T' H_{\tilde{S},\Sigma} H \Delta_{0}
	\right) (dH)
	\right] \\
	&= \E 
	\left[
	{}_{0} F_{1} \left(
	\dfrac{p}{2}; -\dfrac{1}{4} \Delta_{0}T'H_{\tilde{S},\Sigma}
	H_{\tilde{S},\Sigma}'T \Delta_{0}'
	\right)
	\right] \\
	&= \E
	\left[
	\int_{\mathcal{O}(p)} \etr 
	\left(
	i T' H \Delta_{0} 
 	\right)(dH)
	\right] \\
	&= \E
	\left[
	\etr \left(
	i T' \Delta_{0}
	\right)
	\right] \equiv \phi_{\Delta_{0}} (T),
\end{align*} 
which implies 
\(\Delta \stackrel{d}{=} \Delta_{0}\) and 
\(
\Delta = (\Sigma^{1/2} \tilde{S} \Sigma^{1/2})^{-1/2} \Sigma^{1/2} \tilde{X}
= S_{X^{*}}^{-1/2} X^{*}
\), where \(S_{X^{*}}\) is the sample covariance matrix based on 
\(X^{*}\), $\etr(\ast)=\exp(\tr(\ast))$ and ${}_0 F_1 (a_1; X)$ is the Bessel type hypergeometric function of matrix argument (see, for example, \citet[Definition 7.3.1]{Muirhead1982}). 
This then implies we have 
\[
K'X' = K'(X^{*}-\bmu \bones_{N}')' = K'(X^{*})', \quad
S = (1/n)XKK'X' = S_{X^{*}}, \]
consequently
\[
K'X' S^{-1/2} \stackrel{d}{=}K'(X^{*})' S_{X^{*}}^{-1/2} 
\stackrel{d}{=} K'\tilde{X}' \tilde{S}^{-1/2},
\]
\(K'\tilde{X}'\sim N(O, I_{n}\otimes I_{p})\) and
\(n\tilde{S} \sim W_{p} (n, I_{p})\) since 
\(K'\tilde{\bf{X}}_{(j)} \sim N_{n}(\bnull, I_{n})\).
\end{proof}
\end{lemma}
\label{chikuse}
From Lemmas~\ref{lemma-chikuse} and~\ref{lemma-IwashitaKlar}, we have Theorem~\ref{thm-dist-stat}, which is used as the null distribution for testing MVN.
\begin{theorem}
\label{thm-dist-stat}
Let $U$ be the \( n \times p \) matrix defined in \eqref{matrix-Stiefel manifold} and $A \sim W_p(n, I_p)$ be independent of $U$.
Then, we have
\[
UA^{1/2} \sim \mathit{N}(O, I_n\otimes I_p).
\]
\begin{proof}
Applying the matrix polar decomposition to \(K' \tilde{X}'\), we have
\begin{align}
\label{mat-decom}
K' \tilde{X}'
&= K' \tilde{X}' \bigl(\tilde{X} K K' \tilde{X}'\bigr)^{-1/2}
   \bigl(\tilde{X} K K' \tilde{X}'\bigr)^{1/2} \nonumber \\
&= \tilde{Z} (\tilde{Z}' \tilde{Z})^{-1/2} (\tilde{Z}' \tilde{Z})^{1/2},
\end{align}
where $\tilde{Z} = K' \tilde{X}' \tilde{S}^{-1/2}$, with $\tilde{X}$ and $\tilde{S}$ defined in Lemma~\ref{lemma-IwashitaKlar}.
The random matrices $\tilde{Z}(\tilde{Z}'\tilde{Z})^{-1/2}$ and $\tilde{Z}'\tilde{Z}=n\tilde{S}$ in \eqref{mat-decom}
are distributed as $\mathcal{U}_{n,p}$ and $W_p(n,I_p)$, respectively,
and are independent by Lemma~\ref{lemma-chikuse}.
Finally, note that $\tilde{Z} \stackrel{d}{=} Z$ by Lemma~\ref{lemma-IwashitaKlar}, where $Z$ is defined in \eqref{scaled residuals matrix}. 
The proof is complete.\end{proof}
\end{theorem}
Using Theorem~\ref{thm-dist-stat}, we propose a necessary test of MVN.
The null hypothesis \eqref{test-normality} would be obtained by just testing
\begin{align}
\label{test02}
H_0^{\ast} : UA^{1/2} \text{ has a matrix variate normal distribution }
N(O, I_n\otimes I_p),
\end{align}
against the alternative hypothesis $\neg H_0^{\ast}$.
\begin{remark}
As highlighted by \cite{liang2000testing,liang2019testing}, rejection of $H_0^{\ast}$ in \eqref{test-normality} implies rejection of $H_0$ in \eqref{test02}.
On the other hand, note that acceptance of $H_0^{\ast}$ does not imply that $H_0$ is true.
This is precisely why the test based on $H_0^{\ast}$ is called a ``necessary test" for $H_0$.
\end{remark}
We perform univariate normality tests on the $np$ elements of $UA^{1/2}$ in \eqref{test02} and reject $H_0$ if $H_0^{\ast}$ is rejected.
The proposed test based on Theorem~\ref{thm-dist-stat} involves randomness arising from the Wishart matrix transformation; consequently, the decision for a given sample may differ across realizations.
\cite{cuesta2009projection} proposed uniformity tests for directional data based on random projections and mitigated the effect of randomness using Bonferroni's correction.
To stabilize our test, we also incorporate Bonferroni's correction in the same manner as \cite{cuesta2009projection}.
The testing procedure is given as follows.
\begin{algorithm}[H]
\caption{Testing procedure based on random matrix transformations}
\label{algorithm}
\begin{algorithmic}[1]
\STATE Compute the matrix $U$ in \eqref{matrix-Stiefel manifold}.
\STATE Generate independent random matrices
      $A_i \sim W_p(n, I_p)$, $i=1,\dots,m$.
\STATE For each $i = 1, \ldots, m$, compute the transformed data $UA_i^{1/2}$.
\STATE Perform a univariate normality test on the $np$ entries of $UA_i^{1/2}$
       and obtain the corresponding $p$-value $P_i$.
\STATE Reject $H_0$ if $\min_{1 \le i \le m} P_i \le \alpha / m$.
\end{algorithmic}
\end{algorithm}

In Step~4 of Algorithm~\ref{algorithm}, the Anderson--Darling (AD) test and the Shapiro--Wilk (SW) test are applied to assess univariate normality for the entries of $U A_i^{1/2}$. We denote the corresponding methods by $\mathrm{AD}_m$ and $\mathrm{SW}_m$, respectively.

\section{Simulations}
\label{sec:04}
We examine the empirical Type I error rate and power of the proposed test.
We compare our tests $\mathrm{AD}_1$, $\mathrm{AD}_3$, $\mathrm{AD}_5$, $\mathrm{SW}_1$, $\mathrm{SW}_3$, and $\mathrm{SW}_5$ with the classical Mardia and Baringhaus--Henze (BH) tests implemented in Mathematica's built-in functions, where the Method option is set to ``Asymptotic."
Mardia's skewness and kurtosis tests are denoted by $b_M^{(1)}$ and $b_M^{(2)}$, respectively.
Table~1 provides the empirical Type~I error rates obtained from
$10^6$ Monte Carlo replications at the nominal significance level
$\alpha = 0.05$, for $p = 2$, sample sizes $N = 10, 20, 30$, and
$m = 1, 3, 5$.
Without loss of generality, we assume $\vec{\mu}=\vec{0}$ and $\Sigma=I_p$
in the simulation study, as implied by Lemma~2.
\begin{table}[H]
\centering
\small
\caption{Empirical Type I Error rate(in percent) against the multivariate normal distribution.}
\begin{tabular}{cc ccccccc cc}
\toprule
$p$ & $n$
& AD$_1$ & AD$_3$ & AD$_5$
& SW$_1$ & SW$_3$ & SW$_5$
& BH & $b^{(1)}_{M}$ & $b^{(2)}_{M}$ \\
\midrule

\multirow{3}{*}{2}
 & 10 & 4.83 & 4.53 & 4.52 & 5.01 & 4.74 & 4.64 & 2.26 & 1.11 & 0.00 \\
 & 20 & 4.93 & 4.86 & 4.87 & 5.00 & 4.94 & 4.88 & 3.75 & 3.60 & 0.25 \\
 & 30 & 4.98 & 4.90 & 4.98 & 4.97 & 4.92 & 4.91 & 4.23 & 4.35 & 0.86 \\
\midrule

\multirow{3}{*}{3}
 & 10 & 4.90 & 4.74 & 4.70 & 5.02 & 4.90 & 4.84 & 1.53 & 0.21 & 2.33 \\
 & 20 & 4.99 & 4.96 & 4.99 & 4.97 & 4.96 & 4.94 & 3.27 & 2.89 & 3.31 \\
 & 30 & 5.03 & 4.96 & 5.07 & 4.95 & 4.92 & 4.95 & 3.98 & 4.03 & 3.50 \\
\bottomrule
\end{tabular}
\end{table}
Our tests maintain Type I error rates close to the nominal significance level even in small-sample settings.
When $N=10$, the test tends to become conservative as $m$ increases; however, this effect is substantially reduced as $N$ increases, and the nominal significance level is well controlled.

To assess power, we consider four classes of alternatives: symmetric heavy-tailed, symmetric light-tailed, skewed, and bimodal distributions. 
Since the Mardia and BH tests fail to control the Type I error rate in finite samples, we focus only on our proposed test to examine its power.
We consider the following $p$-dimensional alternative distributions:
\begin{enumerate}
\renewcommand{\labelenumi}{(\arabic{enumi})}
\item[Model 1] : Multivariate $t$ distribution with $3$ degrees of freedom.
\item[Model 2] : Multivariate Pearson Type~II distribution with shape parameter $m=0$.
\item[Model 3] : Multivariate log-normal distribution.
\item[Model 4] : A mixture of two multivariate normal distributions with density
\begin{align*}
\frac{1}{2} N_p(\bm{0}, I_p) + \frac{1}{2}  N_p(\bm{0}, \Sigma),
\end{align*}
where $\Sigma = 8\bigl(\rho I_p + (1-\rho)\bm{1}_p \bm{1}_p^{\prime}\bigr)$ with $\rho = 1/2$.
\end{enumerate}
For each distribution, the characteristics of kurtosis and skewness are summarized in Table~\ref{table-power-dim2} of \cite{mecklin2005}.
Models 1 and 2 belong to the class of elliptical distributions, whereas Models 3 and 4 are non-elliptical. 
Tables~\ref{table-power-dim2} and~\ref{table-power-dim3} present the empirical power obtained from $10^6$ Monte Carlo simulations for the cases $p=2$ and $p=3$, respectively.
We can confirm that the power increases as the sample size $N$ increases.
Furthermore, in most cases, $SW_m$ exhibits higher power than $AD_m$.
\begin{table}[H]
\centering
\footnotesize
\caption{Empirical power of the tests (in percent) for $p=2$ and $\alpha=0.05$}
\begin{tabular}{ll cccccc}
\toprule
Alt. Dist. & $n$
& AD$_1$ & AD$_3$ & AD$_5$
& SW$_1$ & SW$_3$ & SW$_5$ \\
\midrule
\multirow[t]{3}{*}{$t_3$}
 & 10  & 10.4 & 8.91 & 8.18 & 11.3 & 10.7 & 10.3  \\
 & 20  & 27.4 & 24.2 & 22.8 & 30.9 & 30.1 & 29.7  \\
 & 30  & 44.6 & 39.9 & 37.9 & 49.7 & 47.9 & 46.9  \\
\midrule
\multirow[t]{3}{*}{Mixture}
 & 10  & 8.24 & 6.67 & 5.96 & 8.83 & 7.87 & 7.27  \\
 & 20  & 17.3 & 13.8 & 12.4 & 18.2 & 16.0 & 15.0  \\
 & 30  & 27.7 & 22.2 & 20.0 & 27.9 & 24.1 & 22.5  \\
\midrule
\multirow[t]{3}{*}{Lognormal}
 & 10  & 46.0 & 45.7 & 44.5 & 48.1 & 48.9 & 48.1  \\
 & 20  & 92.7 & 92.9 & 92.5 & 93.7 & 94.5 & 94.2  \\
 & 30  & 99.4 & 99.5 & 99.4 & 99.5 & 99.6 & 99.6  \\
\midrule
 \multirow[t]{3}{*}{Type {II}}
 & 10  & 5.01 & 3.33 & 2.67 & 4.57 & 2.80 & 2.18  \\
 & 20  & 9.56 & 5.88 & 4.54 & 8.08 & 4.03 & 2.68  \\
 & 30  & 17.4 & 11.0 & 8.52 & 16.6 & 8.88 & 5.97  \\
\bottomrule
\end{tabular}
\label{table-power-dim2}
\end{table}
\begin{table}[H]
\centering
\footnotesize
\caption{Empirical power of the tests (in percent) for $p=3$ and $\alpha=0.05$}
\begin{tabular}{ll cccccc}
\toprule
Alt. Dist. & $n$
& AD$_1$ & AD$_3$ & AD$_5$
& SW$_1$ & SW$_3$ & SW$_5$ \\
\midrule
\multirow[t]{3}{*}{$t_3$}
 & 10  & 8.44 & 8.04 & 7.82 & 8.99 & 9.29 & 9.31 \\
 & 20  & 26.2 & 25.1 & 24.5 & 29.3 & 31.0 & 31.4 \\
 & 30  & 47.7 & 45.4 & 44.3 & 52.9 & 53.8 & 54.0 \\
\midrule
\multirow[t]{3}{*}{Mixture}
 & 10  & 7.16 & 6.48 & 6.22 & 7.57 & 7.36 & 7.28 \\
 & 20  & 17.4 & 14.6 & 13.7 & 18.2 & 16.7 & 16.4 \\
 & 30  & 28.2 & 24.7 & 23.2 & 28.6 & 27.0 & 26.0 \\
\midrule
\multirow[t]{3}{*}{Lognormal}
 & 10  & 49.8 & 54.8 & 55.8 & 51.0 & 56.9 & 58.2 \\
 & 20  & 97.1 & 98.0 & 98.2 & 97.5 & 98.5 & 98.7 \\
 & 30  & 99.9 & 100.0 & 100.0 & 99.9 & 100.0 & 100.0 \\
\midrule
 \multirow[t]{3}{*}{Type {II}}
 & 10  & 4.69 & 3.73 & 2.68 & 4.41 & 3.32 & 2.17 \\
 & 20  & 7.72 & 5.49 & 4.61 & 6.29 & 3.68 & 2.74 \\
 & 30  & 13.6 & 9.45 & 7.93 & 12.3 & 7.31 & 5.43 \\
\bottomrule
\end{tabular}
\label{table-power-dim3}
\end{table}
Finally, we apply our test $SW_1$ to Fisher's \textit{Iris setosa} data set at the significance level $\alpha = 0.05$. 
For all 150 observations consisting of the three species (\textit{Iris setosa}, \textit{Iris versicolor}, and \textit{Iris virginica}), the test rejects the null hypothesis for the four measurements with a $p$-value of 0.004.
Across 500 repeated executions of the test, the proportion of rejections is 0.996. 
Rejecting multivariate normality for this data set agrees with findings reported in \cite{zhou2014powerful}.

\section*{Acknowledgments}
This work was supported by JSPS KAKENHI (Grant Number 25K17300).

\bibliographystyle{elsarticle-harv}
\bibliography{ref}




\end{document}